\documentclass[10pt]{amsart}
\usepackage{amssymb}
\usepackage{epsfig}
\usepackage{url}
\usepackage{setspace}
\theoremstyle{plain}

\newtheorem{thm}{Theorem}[section]
\newtheorem{cor}[thm]{Corollary}

\newtheorem{rem}[thm]{Remark}
\newtheorem{ques}[thm]{Question}
\newtheorem{conj}[thm]{Conjecture}
\newtheorem{exam}[thm]{Example}
\def\cal{\mathcal}
\def\bbb{\mathbb}
\def\op{\operatorname}
\renewcommand{\phi}{\varphi}

\newcommand{\N}{\bbb{N}}
\newcommand{\Z}{\bbb{Z}}
\newcommand{\Q}{\bbb{Q}}

\begin{document}

\title[On some Diophantine systems]{On some Diophantine systems involving symmetric polynomials}
\author{Maciej Ulas}


\begin{abstract}
Let $\sigma_{i}(x_{1},\ldots, x_{n})=\sum_{1\leq k_{1}<k_{2}<\ldots <k_{i}\leq n}x_{k_{1}}\ldots x_{k_{i}}$ be the $i$-th elementary symmetric polynomial. In this note we generalize and extend the results obtained in a recent work of Zhang and Cai \cite{ZC,ZC2}. More precisely, we prove that for each $n\geq 4$ and rational numbers $a, b$ with $ab\neq 0$, the system of diophantine equations
\begin{equation*}
  \sigma_{1}(x_{1},\ldots, x_{n})=a, \quad \sigma_{n}(x_{1},\ldots, x_{n})=b,
\end{equation*}
has infinitely many solutions depending on $n-3$ free parameters. A similar result is proved for the system
\begin{equation*}
  \sigma_{i}(x_{1},\ldots, x_{n})=a, \quad \sigma_{n}(x_{1},\ldots, x_{n})=b,
\end{equation*}
with $n\geq 4$ and $2\leq i< n$. Here, $a, b$ are rational numbers with $b\neq 0$.

We also give some results concerning the general system of the form
\begin{equation*}
  \sigma_{i}(x_{1},\ldots, x_{n})=a, \quad \sigma_{j}(x_{1},\ldots, x_{n})=b,
\end{equation*}
with suitably chosen rational values of $a, b$ and $i<j<n$. Finally, we present some remarks on the systems involving three different symmetric polynomials.
\end{abstract}

\maketitle

\section{Introduction}\label{Section1}

In 1996 A. Schinzel proved that for every $k$, there exist infinitely
many primitive sets of $k$ triples of positive integers having the same sum and
the same product \cite{Sch}. This result solved the problem D16 from Guy's book \cite{Guy}. Recently, in an interesting note \cite{ZC2} Zhang and Cai have generalized this result to $n$-tuples. To be more precise, they proved that for each $n\geq 4$ and each $k$, there exist infinitely many $k$ $n$-tuples with the same sum and the same product. They obtained this result by proving that the system
\begin{equation}\label{mainsys1}
  \sigma_{1}(x_{1},\ldots, x_{n})=a,  \quad \sigma_{n}(x_{1},\ldots, x_{n})=b,
\end{equation}
with $a=b=2n$ has infinitely many rational solutions. They further generalizes this result in \cite{ZC} and proved that for every $k$, there exist infinitely many primitive sets of $k$ $n$-tuples of positive integers with the same second elementary symmetric function value and the same product. In fact they proved that the Diophantine system
\begin{equation*}
  \sigma_{2}(x_{1},\ldots, x_{n})=a,  \quad \sigma_{n}(x_{1},\ldots, x_{n})=b,
\end{equation*}
has infinitely many solutions in positive rational numbers $x_{i},\;i=1,\ldots,n$ for each $n\geq 3$ and $a=\frac{1}{2}(3n^2-n-2),\;b=2n$. In both cases their solutions satisfy $x_{1}=\ldots=x_{n-3}=1$. We thus see that those solutions are rather special. Their solutions are parameterized by rational points on a certain elliptic curve (depending on the integer $n$) defined over the field $\Q$. We also should note that the result from \cite{ZC2} follows from Schinzel's work. This was noted by Zieve in \cite{Zieve}.

When a Diophantine problem has infinitely many solutions in integers or rational numbers a natural question arises: is it possible to show the existence of rational parametric solutions, i.e. solutions in polynomials or in rational functions? Moreover, one can ask for which values of $a, b$ with $b\neq 0$ the system (\ref{mainsys1}) has solutions in rational numbers (not necessarily positive). The next question which comes to mind is whether the more general Diophantine system of the form
\begin{equation}\label{mainsys2}
  \sigma_{i}(x_{1},\ldots, x_{n})=a, \quad \sigma_{n}(x_{1},\ldots, x_{n})=b,
\end{equation}
has solutions for given $a, b$ and sufficiently large $n$ and $2\leq i<n$. From the geometric point of view the algebraic variety defined by the system (\ref{mainsys2}) as the intersection of varieties of dimension $n-1$ is a variety of dimension $n-2$ and degree $n+i-1$. Therefore, it is very likely that there exist rational curves on this variety. This easy observation, as well as the curiosity, whether the results of Zhang and Cai can be improved were the main motivation to write this paper.

Let us describe the content of the paper in some details. In section \ref{Section2} we prove that for each $a, b$ with $ab\neq 0$ the system (\ref{mainsys1}) has infinitely many rational parametric solutions depending on $n-3$ free parameters. Similar result is proved in section \ref{Section3} for the general system (\ref{mainsys2}) provided $2\leq i<n$ with $n\geq 4$ and $b\neq 0$. In section \ref{Section4} we investigate the general system of the form
\begin{equation}\label{mainsys3}
  \sigma_{i}(x_{1},\ldots, x_{n})=a, \quad \sigma_{j}(x_{1},\ldots, x_{n})=b,
\end{equation}
with $i<j<n$ and suitable chosen values of $a, b$. We also suggest that our methods are strong enough to tackle some quite general Diophantine systems involving not necessarily symmetric polynomials. Finally, in the last section we give some remarks concerning the systems containing three different symmetric polynomials. In particular we prove that for each $n\geq 5$ there exist rational numbers $a, b, c$ such that the system
\begin{equation*}
\sigma_{1}(x_{1},\ldots, x_{n})=a,\quad \sigma_{2}(x_{1},\ldots, x_{n})=b,\quad \sigma_{3}(x_{1},\ldots, x_{n})=c
\end{equation*}
has infinitely many solutions in rational numbers $x_{1},\ldots, x_{n}$.

\begin{rem}
{\rm All computations in this paper were performed with the help of Mathematica program \cite{Wol}.}
\end{rem}

\section{Solutions of the system (\ref{mainsys1})}\label{Section2}

The aim of this section is the construction of a general solution of the system (\ref{mainsys1}). More precisely, we prove the following result:

\begin{thm}\label{thm1}
Let $n\geq 4$ and let $a, b$ be given non-zero rational numbers. Then the system {\rm (\ref{mainsys1})} has infinitely many rational parametric solutions depending on $n-3$ free parameters.
\end{thm}
\begin{proof}
First of all we note that it is enough to assume that $a, b$ are integers with $ab\neq 0$. Indeed, let $A, B$ be rational numbers with a common denominator $D$ and let us suppose that $\overline{X}=(x_{1},\ldots,x_{n})$ is a solution of the system (\ref{mainsys1}) with $a=AD, b=BD^{n}$ and $x_{i}\neq 0$ for $i=1,\ldots,n$. We immediately deduce that $\overline{Y}=\frac{\overline{X}}{D}=\Big(\frac{x_{1}}{D},\ldots,\frac{x_{n}}{D}\Big)$ is the solution of the system
$\sigma_{1}(\overline{Y})=A,\;\sigma_{n}(\overline{Y})=B$.

We first prove our result in the case $n=4$ and then deduce the solution of (\ref{mainsys1}) for all $n\geq 5$. We thus work with the system
\begin{equation}\label{i=1:n=4}
x_{1}+x_{2}+x_{3}+x_{4}=a,\quad x_{1}x_{2}x_{3}x_{4}=b.
\end{equation}
Eliminating $x_{4}$ from the first equation in (\ref{i=1:n=4}) we are left with the equation
\begin{equation}\label{special}
x_{1}x_{2}x_{3}(a-x_{1}-x_{2}-x_{3})=b.
\end{equation}
In order to show that this equation has infinitely many solutions for any given $a, b$ with $ab\neq 0$ we put $x_{3}=-4bt^{2}x_{1}$. After this substitution the equation (\ref{special}) takes the form
\begin{equation*}
F_{P}(Q)=b(4t^2P^2Q^2-4 t^2(a+(4bt^2-1)P)P^2Q-1)=0,
\end{equation*}
where in order to shorten the notation we put $x_{1}=P, x_{2}=Q$. It is enough to show that the set of $P\in\Q(t)$ for which the above equation has a solution (with respect to $Q$) in  $\Q(t)$ is infinite. Equivalently the square-free part of the discriminant $\Delta(P)$ of the polynomial $F_{P}$ should be a square in $\Q(t)$. This leads us to the curve
\begin{equation*}
\cal{C}_{1}:\;S^2=t^2(4bt^2-1)^2P^4+2at^2(4bt^2-1)P^3+a^2t^2P^2+1.
\end{equation*}
The discriminant of the quartic polynomial defining the curve $\cal{C}_{1}$ is of the form $16t^6(4bt^2-1)^4(16+(a^4-128b)t^2+256b^2t^4)$ and is non-zero as an element of $\Q(t)$. Thus the curve $\cal{C}_{1}$ is smooth. Let us also note that the $\Q(t)$-rational point $U=(0,1)$ lies on the curve $\cal{C}_{1}$. If we treat $U$ as a point at infinity and use the method described in \cite[p. 77]{Mor}, we conclude that $\cal{C}_{1}$ is birationally equivalent to the elliptic curve $\cal{E}_{1}$ given by the Weierstrass equation
\begin{equation*}
\cal{E}_{1}:\;Y^2=X^3+AX+B,
\end{equation*}
where $A, B\in\Q(t)$ are given by
\begin{align*}
A&=-27 t^2 (12 + (a^4-96 b)t^2 + 192 b^2 t^4),\\
B&=54 a^2 t^4 (18 + (a^4-144 b) t^2 + 288 b^2 t^4).
\end{align*}
We do not present the explicit equations for the coordinates of the map $\phi:\;\cal{C}_{1}\rightarrow\cal{E}_{1}$ because they are quite complicated. In order to prove that the group $\cal{E}_{1}(\Q(t))$ is infinite we note that $\cal{E}_{1}$ contains the point
\begin{equation*}
V=\phi((0,-1))=(X,Y)=(-6a^2t^2, -54at^2(4bt^2-1)).
\end{equation*}
In order to show that $V$ is of infinite order we put $t=1$. We note that the $X$-th coordinate of specialization at $t=1$ of the point $[2]V$ has the form
\begin{equation*}
X_{[2]V_{1}}=\frac{3(4 + 3 a^4 - 32 b + 64 b^2) (12 + a^4 - 96 b +192 b^2)}{16 a^2 (4b-1)^2}.
\end{equation*}
Quick computation reveals that the remainder of the division of the numerator by the denominator with respect to $b$ is equal to $9a^{8}$ and thus is non-zero provided $a\neq 0$. Invoking now the Nagell-Lutz theorem (see \cite[p. 78]{Sko}) we get that $V_{1}$ is of infinite order on the curve $\cal{E}_{1,1}$ - the specialization of $\cal{E}_{1}$ at $t=1$, and thus $V$ is of infinite order on the curve $\cal{E}_{1}$.

Now, it is an easy task to obtain the statement of our theorem . For $m=2,3,\ldots $ we compute the point $[m]V=\sum_{i=1}^{m}V$ on the curve $\cal{E}_{1}$; next we calculate the corresponding point $\phi^{-1}([m]V)=(P_{m},S_{m})$ on $\cal{C}_{1}$ and solve the equation $F_{P_{m}}(Q)=0$. We put the calculated roots into the expression for $P$ and get various $\Q(t)$-rational solutions of the equation (\ref{special}) and the system (\ref{i=1:n=4}). For example the point $[2]V$ leads us through the calculation of $\phi^{-1}([2]V)$ and the computation of $Q$ to the following solution of the system (\ref{i=1:n=4}):

\begin{align*}
x_{1}&=\frac{8a(4bt^2-1)}{-64b^2t^4+(a^4+32b)t^2-4},\\
x_{2}&=-\frac{32 a b t^2 \left(4 b t^2-1\right)}{\left(a^2 t-8 b t^2+2\right) \left(a^2 t+8 b t^2-2\right)},\\
x_{3}&=\frac{\left(a^2 t+8 b t^2-2\right)^3}{16 a t \left(4 b t^2-1\right) \left(a^2 t-8 b t^2+2\right)},\\
x_{4}&=a-x_{1}-x_{2}-x_{3}.
\end{align*}

In order to tackle the system
\begin{equation}\label{i=1:general n}
\sigma_{1}(x_{1},\ldots,x_{n})=a,\quad \sigma_{n}(x_{1},\ldots,x_{n})=b
\end{equation}
with $n\geq 5$ we use the following reasoning. Let $x'_{5},\ldots,x'_{n}$ be rational parameters and put $p=\sum_{i=5}^{n}x'_{i}$ and $q=\prod_{i=5}^{n}x'_{i}$. From our reasoning we know that the system (\ref{i=1:n=4}) has infinitely many solutions $(x'_{1,j},\ldots,x'_{4,j}), j\in\N$, depending on one parameter $t$ for $A=a-p$ and $B=\frac{b}{q}$. This immediately implies that for each $j\in\N$ the $n$-tuple of the following form
\begin{equation*}
x_{1}=x'_{1,j},\;x_{2}=x'_{2,j},\;x_{3}=x'_{3,j},\;x_{4}=x'_{4,j},\;x_{i}=x'_{i}\;\mbox{for}\;i=5,\ldots,n
\end{equation*}
solves the system (\ref{i=1:general n}).
\end{proof}

\section{Solutions of the system (\ref{mainsys2})}\label{Section3}

In this section we consider the general system (\ref{mainsys2}) for $2\leq i\leq n$.

\begin{thm}\label{thm2}
Let $n\geq 4,\;2\leq i\leq n$ and $a, b$ be rational numbers with $b\neq 0$. Under these assumptions the system {\rm (\ref{mainsys2})} has infinitely many rational parametric solutions depending on $n-3$ free parameters.
\end{thm}
\begin{proof}
Let us note that we can assume that $i\leq n/2$. Indeed, if $\overline{X}=(x_{1},\ldots,x_{n})$ is non-zero solution of the system (\ref{mainsys2}) with $a=\frac{A}{B}$ and $b=\frac{1}{B}$ then $\overline{Y}=\frac{1}{\overline{X}}=(\frac{1}{x_{1}},\ldots, \frac{1}{x_{n}})$ is solution of the system
$\sigma_{n-i}(\overline{Y})=A,\;\sigma_{n}(\overline{Y})=B$. Indeed, we have
\begin{equation*}
\frac{A}{B}=\sigma_{i}(\overline{X})=\sigma_{i}\Big(\frac{1}{\overline{Y}}\Big)=\frac{\sigma_{n-i}(\overline{Y})}{\sigma_{n}(\overline{Y})}
\end{equation*}
and
\begin{equation*}
\frac{1}{B}=\sigma_{n}(\overline{X})=\sigma_{n}\Big(\frac{1}{\overline{Y}}\Big)=\frac{1}{\sigma_{n}(\overline{Y})}
\end{equation*}
and thus $\sigma_{n}(\overline{Y})=B$ and $\sigma_{n-i}(\overline{Y})=A$. We also note that it is enough to assume that $a, b$ are integers with $b\neq 0$. Indeed, let $A, B$ be rational numbers with common denominator $D$ and let us suppose that $\overline{X}=(x_{1},\ldots,x_{n})$ is a solution of the system (\ref{mainsys2}) with $a=AD^{i}, b=BD^{n}$ and $x_{i}\neq 0$ for $i=1,\ldots,n$. We immediately deduce that $\overline{Y}=\frac{\overline{X}}{D}=\Big(\frac{x_{1}}{D},\ldots,\frac{x_{n}}{D}\Big)$ is the solution of the system
$\sigma_{i}(\overline{Y})=A,\;\sigma_{n}(\overline{Y})=B$.

\bigskip

Let us fix $n\geq 4$ and $2\leq i\leq \frac{n}{2}$. In order to shorten the notation let us put
\begin{equation*}
\overline{X}=(x_{1},x_{2},\ldots, x_{n-3}),\quad x_{n-2}=P,\quad x_{n-1}=Q,\quad x_{n}=R.
\end{equation*}
Thus, we can write our system as $\sigma_{i}(\overline{X},P,Q,R)=a,\;\sigma_{n}(\overline{X},P,Q,R)=b$. Now let us note the following equality
$\sigma_{n}(\overline{X},P,Q,R)=\sigma_{n-3}(\overline{X})PQR$ and
\begin{align*}
\sigma_{i}&(\overline{X},P,Q,R)=\sigma_{i-1}(\overline{X},P,Q)R+\sigma_{i}(\overline{X},P,Q)\\
                              &=(\sigma_{i-2}(\overline{X},P)Q+\sigma_{i-1}(\overline{X},P))R+\sigma_{i-1}(\overline{X},P)Q+\sigma_{i}(\overline{X},P)\\
                              &=(\sigma_{i-3}(\overline{X})P+\sigma_{i-2}(\overline{X}))Q+(\sigma_{i-2}(\overline{X})P+\sigma_{i-1}(\overline{X}))R\\
                              &\quad +(\sigma_{i-2}(\overline{X})P+\sigma_{i-1}(\overline{X}))Q+\sigma_{i-1}(\overline{X})P+\sigma_{i}(\overline{X})\\
                              &=\sigma_{i-3}(\overline{X})PQR+\sigma_{i-2}(\overline{X})(PQ+QR+RP)+\sigma_{i-1}(\overline{X})(P+Q+R)+\sigma_{i}(\overline{X}).
\end{align*}
In these expansion we use the well known convention which says that $\sigma_{m}(\overline{X})$ is defined to be 0 for $m<0$ and is equal to 1 in case of $m=0$. Putting now
\begin{equation*}
u=\sigma_{i-3}(\overline{X}),\quad v=\sigma_{i-2}(\overline{X}),\quad w=\sigma_{i-1}(\overline{X}),\quad t=\sigma_{i}(\overline{X}),\quad m=\sigma_{n-3}(\overline{X})
\end{equation*}
we see that the system (\ref{mainsys2}) is equivalent to
\begin{equation}\label{mainsys2a}
\begin{cases}
\begin{array}{lll}
  uPQR+v(PQ+QR+RP)+w(P+Q+R)+t & = & a, \\
  mPQR & = & b.
\end{array}
\end{cases}
\end{equation}
Let us note that in case of $i=2$ we have $u=0, v=1$. In case of $i=3$ we have $u=1$. Moreover, we note that we can work in the field $\Q(u,v,w)$, where $u, v, w$ are treated as independent variables. In order to solve the above system we use a similar approach as in the proof of Theorem \ref{thm1}. Solving the second equation from (\ref{mainsys2a}) with respect to $P$ and putting the solution into the first equation we are left with the quadratic equation in $Q$:
\begin{equation*}
G_{R}(Q):=mR(Rv+w)Q^2+(mwR^2+(m(t-a)+bu)R+bv)Q+b(Rv+w)=0.
\end{equation*}
It is enough to show that the set of $R\in\Q(u,v,w)$ for which the above equation has a solution (with respect to $Q$) in  $\Q(u,v,w)$ is infinite. Equivalently, the discriminant $\Delta(R)$ of the polynomial $G_{R}$ should be a square in $\Q(u,v,w)$. This leads us to the curve
\begin{equation*}
\cal{C}_{2}:\;S^2=(mwR^2+(m(t-a)+bu)R+bv)^2-4bmR(Rv+w)^2
\end{equation*}
defined over the field $\Q(u,v,w)$. The discriminant of the quartic polynomial defining the curve $\cal{C}_{2}$ is non zero for $b\neq 0$ as an element of the field $\Q(u,v,w)$. This implies that the curve $\cal{C}_{2}$ is smooth. Let us also note that the $\Q(u,v,w)$-rational point $V=(0,bv)$ lies on the curve $\cal{C}_{2}$. If we treat $V$ as a point at infinity and use the method described in \cite{Mor} one more time, we conclude that $\cal{C}_{2}$ is birationally equivalent with the elliptic curve $E$ given by the Weierstrass equation
\begin{equation*}
\cal{E}_{2}:\;Y^2=X^3-27AX+54B,
\end{equation*}
where $A, B\in\Q(u,v,w)$ are given by
\begin{align*}
A&=(bu-(a-t)m)\times\\
    &\quad \quad [-b^3u^3+3b^2((a-t)u^2-8v^3)m-3bu(a-t)^2m^2+(a-t)^3m^3\\
    &\quad \quad\quad  -24bmv((a-t)m-bu)w-24bm^2w^3],\\
B&=b^6 u^6-6b^5u^3((a- t)u^2-6v^3)m+3b^4[(5a^2-10at+5t^2)u^4\\
      &\quad -36(av^3-t)u^2v^3+72 v^6]m^2-4 b^3(a-t)^2u(5(a-t)u^2-27 v^3)m^3\\
      &\quad +3b^2(a-t)^3(5(a-t)u^2 - 12 v^3)m^4 -6bu(a-t)^5m^5+(a - t)^6m^6\\
      &\quad -36bmv((a-t)m - bu)[-b^3u^3+3b^2((a-t)u^2- 4 v^3)m-3bu(a - t)^2m^2+(a-t)^3m^3]w\\
      &\quad +216b^2m^2v^2((a-t)m-bu)^2w^2\\
      &\quad -36bm^2(-b^3u^3+3b^2((a-t)u^2-4v^3)m-3bu(a - t)^2m^2+(a - t)^3m^3)w^3\\
      &\quad +432b^2m^3v((a-t)m -bu)w^4+216b^2m^4w^6.
\end{align*}
We do not present the explicit equations for the coordinates of the map $\psi:\;\cal{C}_{2}\rightarrow \cal{E}_{2}$ and expression for discriminant because they are very complicated. However, let us note that $\Delta(\cal{E}_{2})=-2^{2}3^{9}(A^3-B^2)$.

In order to prove that the group $\cal{E}_{2}(\Q(\overline{X}))$ is infinite we note that $\cal{E}_{2}$ contains the point $V=\psi((0,-bv))=(X,Y)$, where
\begin{align*}
X&=\frac{3}{v^2}(((a-t)m-bu)^2v^2+12bmv^3w+12mv((a-t)m-bu)w^2+12 m^2w^4),\\
Y&=\frac{108m}{v^3}(bv^3+v((a-t)m-bu)m+mw^3)(bv^3+v((a-t)m-bu)m+2mw^3).
\end{align*}
 In order to show that the point $V$ is of infinite order we invoke the well known generalization of the classical Nagell-Lutz theorem. The first generalization states that if on the elliptic curve $E:\;Y^2=X^3+\cal{A}X+\cal{B}$ over the field $\Q(t_{1},\ldots,t_{k})$ of rational functions in $k$ variables with $\cal{A}, \cal{B}\in\Z[t_{1},\ldots,t_{k}]$, there is an integral point $(X,Y)$ which satisfies the condition $Y^2\nmid \Delta(E)$ then $(X,Y)$ is of infinite order in the group $E(\Q(t_{1},\ldots,t_{k}))$. Under the same assumptions on $E$ we also know that the torsion points on the elliptic curve $E$ have coordinates in $\Z[t_{1},\ldots,t_{k}]$ (see \cite[p. 177]{Sil} or \cite[p. 268]{Con}). We use the generalizations mentioned in order to prove that $V$ is of infinite order.

 If $i=2$ then we have that $v=1$ and $X, Y$ are polynomials. We thus consider the curve $\cal{E}_{2,1}$ which is the specialization of the curve $\cal{E}_{2}$ at $v=1$ together with the specialized point $V_{1}$ which comes from $V$. Brute force computation reveals that in our case the remainder of the division of $\Delta(\cal{E}_{2,1})$ with respect to $m$ by the square of the $Y$-th coordinate of the point $V_{1}$ is non zero as an element of the ring of polynomials $\Z[u,w]$. The first generalization of the Nagell-Lutz theorem implies that $V$ is of infinite order on $\cal{E}_{2}$. If $i\geq 3$ then the coordinates of $V$ do not lie in $\Z[u,v,w]$. From the second generalization of the Nagell-Lutz theorem we immediately deduce that $V$ is a point of infinite order on $\cal{E}_{2}$. We thus now that in each case $V$ is of infinite order in the group $\cal{E}_{2}(\Q(u,v,w))$.

Now, it is an easy task to obtain the statement of our theorem. For $k=2,3,\ldots $ we compute the point $[k]V=\sum_{i=1}^{k}V$ on the curve $\cal{E}_{2}$; next we calculate the corresponding point $\phi^{-1}([k]V)=(R_{k},S_{k})$ on $\cal{C}_{2}$ and solve the equation $G_{R_{k}}(Q)=0$. We put the calculated roots into the expression for $P$ and get various $\Q(\overline{X})$-rational solutions of the system (\ref{mainsys2a}) and the system (\ref{mainsys2}).

In order to see an example of solutions of the system (\ref{mainsys2a}) we computed the point $[2]V$ which lead us through the calculation of $\psi^{-1}([2]V)$ and the computation of $P$ to the following solution of the system (\ref{mainsys2a}) and thus the system (\ref{mainsys2}) (with $x_{n-2}=P,\;x_{n-1}=Q,\;x_{n}=R)$:
\begin{align*}
P&=\frac{w(bv^3-mw^3)}{v(bv^3+((a-t)m-bu)vw + 2mw^3)},\\
Q&=-\frac{bv^2(bv^3+((a-t)m-bu)vw + 2mw^3)}{mw^2(2bv^3+((a-t)m-bu)vw + mw^3)},\\
R&=-\frac{w(2bv^3+((a-t)m-bu)vw+mw^3)}{v(bv^3-mw^3)}.
\end{align*}
Summing up, we see that $n$-tuple $(x_{1},\ldots, x_{n})=(x_{1},\ldots, x_{n-3},P,Q,R)$ solves the system (\ref{mainsys2}).
\end{proof}

Having proved the Theorem \ref{thm1} and Theorem \ref{thm2} above, we immediately deduce the following:

\begin{cor}
Let $n\geq 4$ and $1\leq i< n$ and let us put $\overline{X}=(x_{1},\ldots,x_{n-3})$. Then for every positive integer $k$, there exist infinitely many primitive sets of $k$ $n$-tuples of polynomials from $\Z[\overline{X}]$ with the same $i$-th elementary symmetric function value and the same product.
\end{cor}
\begin{proof}
Let the triple $(P_{j}, Q_{j}, R_{j})$ for $j=1,\ldots, k$ be a solution of the system (\ref{mainsys2a}). Let us write
\begin{equation*}
P_{j}=\frac{p_{j}}{d}, \;Q_{j}=\frac{q_{j}}{d},\;R_{j}=\frac{r_{j}}{d}
\end{equation*}
with $\op{gcd}(\{\op{gcd}(p_{j},q_{j},r_{j}):\;j=1,\ldots, k\},d)=1$ and $p_{j}, q_{j}, r_{j}\in\Z[\overline{X}]$. Then we have
\begin{equation}\label{values}
\sigma_{i}(y_{1,j},y_{2,j},\ldots,y_{n,j})=ad^{i},\quad \sigma_{n}(y_{1,j},y_{2,j},\ldots,y_{n,j})=bd^{n},
\end{equation}
for $j=1,\ldots,k$ and
where
\begin{equation*}
y_{j,l}=x_{j}d\quad\mbox{for}\quad j=1,\ldots, n-3,\;l=1,\ldots, k,
\end{equation*}
and $y_{n-2,l}=p_{j}, y_{n-1,l}=q_{l},\;y_{n,l}=r_{l}$. If two sets of solutions
$\{(y_{1,j},\ldots, y_{n,j}), j\leq k\}$ and $\{(y'_{1,j},\ldots, y'_{n,j}), j\leq k\}$ coincide, then from (\ref{values}) we get $d=d'$ and the $n$-tuples coincide.
Because the set of solutions of the system (\ref{mainsys2a}) is infinite we see that for each $k$ we can find $k$ $n$-tuples of polynomials which  satisfy the system (\ref{mainsys2}).
\end{proof}

\section{Some solutions of the system (\ref{mainsys3})}\label{Section4}

In this section we investigate the general system (\ref{mainsys3}). We were unable to prove that for generic choices of $a, b$ the system (\ref{mainsys3}) has infinitely many rational parametric solutions. However, we believe that the following is true:

\begin{conj}\label{conj1}
For any given $1\leq i<j$ and for each pair of rational numbers $a, b$ there exist a positive integer $n$ such that $j<n$ and the system of Diophantine equations
\begin{equation*}
  \sigma_{i}(x_{1},\ldots, x_{n})=a, \quad \sigma_{j}(x_{1},\ldots, x_{n})=b,
\end{equation*}
has infinitely many solutions in rational numbers $x_{i},\;i=1,2,\ldots, n$.
\end{conj}

Although we were not able to prove the above conjecture we show that for suitably chosen polynomial values of $a$ and $b$ the system (\ref{mainsys3}) has parametric solutions. More precisely, we prove the following:

\begin{thm}\label{thm3}
Let $n\geq 3$ and $1\leq i<j\leq n$ be given. Let $t_{1}, \ldots, t_{n}$ be rational parameters. Then the system of Diophantine equations
\begin{equation}\label{specialsys}
  \sigma_{i}(x_{1},\ldots, x_{n})=\sigma_{i}(t_{1},\ldots, t_{n}), \quad \sigma_{j}(x_{1},\ldots, x_{n})=\sigma_{j}(t_{1},\ldots, t_{n}),
\end{equation}
has infinitely many solutions in rational functions $x_{i}\in\Q(t_{1},\ldots,t_{n})$.
\end{thm}
\begin{proof}
In order to shorten the notation we introduce the parameters
\begin{equation*}
\begin{array}{llll}
  \overline{X}=(x_{1},\ldots, x_{n-3}), & P=x_{n-2}, & Q=x_{n-1}, & R=x_{n}, \\
  \overline{T}=(t_{1},\ldots, t_{n-3}), & p=t_{n-2}, & q=t_{n-1}, & r=t_{n}.
\end{array}
\end{equation*}
Thus our system (\ref{specialsys}) can be written as
\begin{equation*}
 \sigma_{i}(\overline{X},P,Q,R)=\sigma_{i}(\overline{T},p,q,r), \quad \sigma_{j}(\overline{X},P,Q,R)=\sigma_{j}(\overline{T},p,q,r).
\end{equation*}
First of all let us note that if $x_{i}=t_{i}$ for $i=1,\ldots, n-3$ then the above system (as a system of two equations in three variables $P, Q, R$ has a solution $P=p, Q=q, R=r$. We thus put $\overline{X}=\overline{T}$. This immediately implies that the values contain only the variables $t_{i}$ for $i=1,\ldots,n-3$, i.e. those free of $p, q, r, P, Q, R$ are cancel. Now using the expression for $\sigma_{i}(\overline{X},P,Q,R)$ given on the beginning of the proof of Theorem \ref{thm2}, we deduce that our system is equivalent to
\begin{equation}\label{specialsys1}
\begin{cases}
\begin{array}{l}
  u_{1}PQR+v_{1}(PQ+QR+RP)+w_{1}(P+Q+R)=\\
  \hskip 3cm u_{1}pqr+v_{1}(pq+qr+rp)+w_{1}(p+q+r), \\
  u_{2}PQR+v_{2}(PQ+QR+RP)+w_{2}(P+Q+R)=\\
  \hskip 3cm u_{2}pqr+v_{2}(pq+qr+rp)+w_{2}(p+q+r),
\end{array}
\end{cases}
\end{equation}
where in order to shorten the notation we put
\begin{equation*}
\begin{array}{lll}
  u_{1}=\sigma_{i-3}(\overline{T}), & v_{1}=\sigma_{i-2}(\overline{T}), & w_{1}=\sigma_{i-1}(\overline{T}), \\
  u_{2}=\sigma_{j-3}(\overline{T}), & v_{2}=\sigma_{j-2}(\overline{T}), & w_{2}=\sigma_{j-1}(\overline{T}).
\end{array}
\end{equation*}
The full system (\ref{specialsys1}), i.e. this one with $u_{i}v_{i}w_{i}\neq 0$ in $\Q(\overline{T})$ for $i=1,2,3$, arises only in case $3\leq i<j$. We thus consider first the case with $(i,j)=(1,2)$.

\bigskip

\noindent Case $(i,j)=(1,2)$. In this case we have $u_{1}=v_{1}=u_{2}=0$ and $w_{1}=v_{2}=1$. Thus, the system (\ref{specialsys1}) represents a curve of genus 0 with known $\Q(\overline{T})$ rational point and thus can be parameterized by rational functions. The parametrization is given by
\begin{align*}
&P=\frac{pu^2+2(r-q)u+2q+2r-p}{u^2+3},\\
&Q=\frac{qu^2+2(p-r)u+2p-q+2r}{u^2+3},\\
&R=\frac{ru^2+2(q-p)u+2p+2q-r}{u^2+3}.
\end{align*}

Now let us consider the general case $2\leq i<j\leq n$. Let us denote here right-hand sides of the equations from (\ref{specialsys1}) by $A$ and $B$ respectively. Eliminating $P$ from the first equation and putting the obtained expression into the second equation we are left with a quadratic equation in $Q$ of the form $aQ^2+bQ+c=0$, where
\begin{align*}
&a=a(R)=\left(u_1 v_2-u_2 v_1\right)R^2+\left(u_1 w_2-u_2 w_1\right)R+v_1 w_2-v_{2}w_{1},\\
&b=b(R)=\left(u_1 w_2-u_2 w_1\right)R^2+\left(A u_2-B u_1-v_2 w_1+v_1 w_2\right)R+A v_2-B v_1,\\
&c=c(R)=-v_{2}w_{1}R^2+\left(A v_2-B v_1\right)R+A w_2-B w_1.
\end{align*}
We thus see that in order to get solutions of our system it is enough to study the curve
\begin{equation*}
\cal{C}:\;S^2=b(R)^2-4a(R)c(R).
\end{equation*}
The degree of the polynomial $b(R)^2-4a(R)c(R)$ is at most four, which implies that the genus of $\cal{C}$ is at least one. The degree drops to three if and only if $u_1 w_2-u_2 w_1=0$. From the definition of our system we know that the curve $\cal{C}$ contains at least one $\Q(\overline{T},p,q,r)$ rational point, i.e. the point
\begin{equation*}
W=(r,(p-q) \left(r^2 \left(u_2 v_1-u_1 v_2\right)+r \left(u_2 w_1-u_1 w_2\right)+v_2 w_1-v_1 w_2\right).
\end{equation*}
Of course, treating the point $W$ on $\cal{C}$ as a point at infinity and using a standard method we can see that $\cal{C}$ is birationally equivalent to the elliptic curve given by the Weierstrass equation of the form
\begin{equation*}
\cal{E}:\;Y^2=X^3+\cal{A}X+\cal{B},
\end{equation*}
where $\cal{A}, \cal{B}$ are very complicated polynomials dependent on $u_{i}, v_{i}$ for $i=1,2,3$ and variables $p, q, r$. One can prove that the point $\chi (-W)$ (where $\chi:\;\cal{C}\rightarrow \cal{E}$ is the birational map) which lies on $\cal{E}$ is of infinite order. This is checked with the generalization of the Nagell-Lutz theorem which was used at the end of the proof of Theorem \ref{thm1}. We do not present explicit equations and computations because there are not very enlightening. However, let us describe how one can produce the point $[2]W$ (and thus a parametric solution of our problem) without invoking the expression for $\cal{E}$ nor the map $\chi$. In order to do that, we can find a parabola $S=a_{1}R^2+b_{1}R+c_{1}$ (here $c_{1}$ is the $S$-th coordinate of the point $W$) which is tangent to the point $W$ with multiplicity three (or two in the case when $\op{deg}_{R}(b(R)^2-4a(R)c(R))=3$ and then $a_{1}=0$). Then the polynomial equation
\begin{equation*}
(a_{1}R^2+b_{1}R+c_{1})^2=b(R)^2-4a(R)c(R)
\end{equation*}
is of degree four with a triple root at $R=r$. This implies that the fourth root, say $R'$, must be rational and we get $[2]W=(R',a_{1}R'^2+b_{1}R'+c_{1})$. However, even the expression for $R'$ is very complicated and we do not see any reason to present it explicitly.

Let us also note that the cases $(i,j)=(1,j)$ with $j\geq 3$ and $i=2$ are also covered by our reasoning. Indeed, if $(i,j)=(1,j)$ with $j\geq 3$ then we have $u_{1}=v_{1}=0$ and $u_{2}v_{2}w_{2}\neq 0$ (in the case $j=3$ we have $u_{2}=1\neq 0$) and our point $W$ is still of infinite order on $\cal{C}$ (which was checked via the image of $\chi(-W)$ on $\cal{E}$). One can also check that a similar reasoning works for $i=2$ with $u_{1}=0,u_{2}=1$.
\end{proof}

\begin{cor}
Let $n\geq 6$ and $1\leq i<j\leq n$ and let us put $\overline{X}=(x_{1},\ldots,x_{n-3})$. Then for every positive integer $k$, there exist infinitely many primitive sets of $k$ $n$-tuples of polynomials from $\Z[\overline{X}]$ with the same $i$-th elementary symmetric function value and the same $j$-th elementary symmetric function value.
\end{cor}

We present now two examples of the method employed to prove Theorem \ref{thm3} in action. In the first one the curve $\cal{C}$ is cubic (which is very convenient). In the second one the polynomial on the right-hand side of the equation defining $\cal{C}$ is genuine a quartic.

\begin{exam}\label{exam:cubic}
{\rm
 Let us consider the system (\ref{mainsys3}) with $i=3, j=4, n=4$, i.e. we are interested in quadruplets of rational numbers with the same third symmetric function and the same product. This system was not solved before. We put $x_{1}=t$ and as in the proof of the Theorem \ref{thm3} $x_{2}=P, x_{3}=Q, x_{4}=R$. In this case the quartic reduces to a cubic and $\cal{C}$ takes a particularly simple form:
\begin{equation*}
\cal{C}:\;S^2=(pqr-(pq+pr+qr)R)^2-4pqrR^3,
\end{equation*}
with the point $W=(r,(p-q)r^2)$. The point $[2]W$ is given by
\begin{equation*}
[2]W=\Big(\frac{p q (p-r) (r-q)}{r (p-q)^2},\frac{pq(p^2q^2(p+q)-6p^2q^2r+3pq(p+q)r^2-(p^2+q^2)r^3)}{r(p-q)^3}\Big)
\end{equation*}
and the solution of our system is of the form
\begin{equation*}
P=\frac{pr(p-q)(q-r)}{q(p-r)^2},\quad Q=\frac{qr(q-p)(p-r)}{p(q-r)^2},\quad R=\frac{pq(p-r)(r-q)}{r(p-q)^2}.
\end{equation*}
Computing the expressions for $[k]W,\;k=3,4,\ldots$ we get infinitely many parametric solutions of our system. Let us also note that $W'=(0,pqr)$ lies on the curve $\cal{C}$. However, it cannot be used in order to find parametric solutions of our system because it is of finite order.
}
\end{exam}

\begin{exam}\label{exam:quartic}
{\rm Let us consider the system (\ref{mainsys3}) with $i=3, j=4, n=5$. In order to shorten the notation we put $x_{1}=x_{2}=1, x_{3}=P, x_{4}=Q, x_{5}=R$. Next, we put $a=\sigma_{3}(1,1,1,2,r)$ and $b=\sigma_{4}(1,1,1,2,r)$. The curve $\cal{C}$ takes the form
\begin{equation*}
\cal{C}:\;S^2=(2R^2-11(r+1)R+5r-3)^2-4(R^2+(5r-3)R+7r+2)(3R^2+2R+1)
\end{equation*}
and the point $W$ is given by $W=(r, 1+2r+3r^2)$. The point $R$-th coordinate of the point $[2]W$ is given by
\begin{equation*}
R=\frac{102 r^5-525 r^4+1138 r^3-1164 r^2+456 r+65}{171 r^4-740 r^3+1356 r^2-1200 r+449}
\end{equation*}
and leads us to the expressions for $P, Q$ of the form:
\begin{equation*}
P=\frac{11 r^4-43 r^3+79 r^2-73 r+32}{17 r^4-88 r^3+188 r^2-192 r+81},\quad Q=\frac{29 r^4-138 r^3+271 r^2-254
   r+98}{6 r^4-28 r^3+61 r^2-66 r+33}.
\end{equation*}
Summing up, we see that the quintiple  $(x_{1},x_{2}, x_{3},x_{4},x_{5})=(1,1,P,Q,R)$ solves the system (\ref{mainsys3}) with $i=3, j=4, n=5$ with $a=\sigma_{3}(1,1,1,2,r)$ and $b=\sigma_{4}(1,1,1,2,r)$. Computation of $[m]W$ for $m=3, 4,\ldots$ give us infinitely many solutions.

}
\end{exam}

We end this section with some remarks concerning a more general system than the one considered previously. Let $n\geq 3$ be a given positive integer and for $k\leq n$ let us define the following subset of $\N^{k}$:
\begin{equation*}
I_{k}=\{(i_{1},i_{2},\ldots, i_{k})\in\N^{k}:\;1\leq i_{1}<i_{2}<\ldots <i_{k}\leq n\}
\end{equation*}
and for $\iota\in I_{k}$ let us put $X_{\iota}=x_{i_{1}}x_{i_{2}}\ldots x_{i_{k}}$. Using this notation we can consider the system of the form
\begin{equation}\label{generalsys}
\sum_{k=1}^{n}\sum_{\iota\in I_{k}}a_{\iota}X_{\iota}=a,\quad \sum_{k=1}^{n}\sum_{\iota\in I_{k}}b_{\iota}X_{\iota}=b,
\end{equation}
where $a, b, a_{\iota}, b_{\iota}$ are given rational numbers. It is clear that the systems considered in the previous sections are very special cases of this general one. One can thus ask, what can be proved about the solvability of (\ref{generalsys})? We expect that for all choices of $a_{\iota}, b_{\iota}$ there should exist $a, b$ such that the above system has infinitely many solutions. In fact this can be proved using exactly the same type of reasoning which was used in the proof of Theorem \ref{thm3}. We do not give all the details, but let us sketch the main lines of the proof. First of all, let us note that if the equations from (\ref{generalsys}) are independent, i.e. they contain different sets of variables, there is nothing to prove. In fact, in order to have something nontrivial to solve we can assume that at least three common variables appear in both equations, say $P, Q, R$. Let $\overline{T}$ denote the set of remaining variables. Let us note that in each variable the degree of both equations for (\ref{generalsys}) is equal to one. Eliminating now $P$ from the first equation and putting into the second equation we get a quadric polynomial, say $F$, in the variable $Q$. Of course it may happen that $F$ reduces to a linear function in $Q$, but then we just solve for the linear equation $F(Q)=0$ and get solutions of our system. We thus can assume that
$\op{deg}_{Q}F=2$. Let us write $F(Q)=A_{2}Q^2+A_{1}Q+A_{0}$. The coefficients of $F$ are polynomials in the variable $R$ with coefficients dependent on the set of variables $\overline{T}$. However, the most important thing is that the degree of $A_{i}$ for $i=0,1,2$ is at most two. We know that $F(Q)=0$ has rational roots if and only if $A_{1}(R)^2-4A_{2}(R)A_{0}(R)$ is a square. We thus meet with the problem of finding rational points on the curve
\begin{equation*}
\cal{C}:\;S^2=A_{1}(R)^2-4A_{2}(R)A_{0}(R).
\end{equation*}
The genus of the curve $\cal{C}$ is at most one. Having a rational point $U$ of infinite order on $\cal{C}$ guarantees the existence of infinitely many rational solutions of our initial system (\ref{generalsys}). In order to guarantee the existence of a rational point on $\cal{C}$ one can use the method which was exploited in the proof of Theorem \ref{thm3}, i.e. taking $a, b$ to be values of polynomials on the left with different set of variables.

It is very unlikely that (\ref{generalsys}) has parametric solutions for all choices of $a_{\iota}, b_{\iota}\in\Q$ and rational parameters $a, b$. However, the results presented in Theorem \ref{thm1} and Theorem \ref{thm2} show that sometimes we have solutions for generic parameters $a, b$. This leads us to the following:

\begin{ques}
Let $n\geq 3 $ be given. What conditions on $a_{\iota}, b_{\iota}$ and $n$ guarantee that the system {\rm (\ref{generalsys})} with given rational parameters $a, b$ has infinitely many rational parametric solutions?
\end{ques}

\section{Some remarks on related systems of Diophantine equations}\label{Section5}

The referee of the paper \cite{ZC2} asked the authors whether there exist infinitely many $n$-tuples of positive integers with
the same sum, the same product, and the same value of the second elementary symmetric polynomial for $n\geq 4$. Although we could not solve this problem, we are able to prove the following result.

\begin{thm}
There are infinitely many triples $a,b,c$ of rational numbers with $c\neq 0$ such that the system
\begin{equation}\label{finitesys}
\sigma_{i_{1}}(P,Q,R,S)=a,\quad \sigma_{i_{2}}(P,Q,R,S)=b,\quad \sigma_{i_{3}}(P,Q,R,S)=c
\end{equation}
with $(i_{1},i_{2},i_{3})=(1,2,3), (1,3,4)$ has infinitely many solutions in rational numbers.
\end{thm}
\begin{proof}
We start with $(i_{1},i_{2},i_{3})=(1,2,3)$ and consider the system (\ref{finitesys}). We compute the Gr\"{o}bner basis of the ideal, say $I_{1}$, containing the equations defining (\ref{finitesys}) with the order $P<Q<R<S$. We get $\op{GB}(I_{1})=\{\sigma_{1}(P,Q,R,S)-a, F_{1}(Q,R,S),F_{2}(R,S)\}$,
where
\begin{align*}
F_{1}(Q,R,S)&=\Big(Q + \frac{1}{2}(R+S-a)\Big)^2-\frac{a^2}{4}-\frac{1}{2}a(R+S)+b+\frac{3}{4}(R+S)^2-RS,\\
  F_{2}(R,S)&=-c + a R S + b (R + S) - 2 R S (R + S) - a (R + S)^2 + (R + S)^3.
\end{align*}
We thus see that each quadruple $(P,Q,R,S)$ satisfying (\ref{finitesys}) gives us a point $(R,S)$ on the curve $C_{1}:\;F_{2}(R,S)=0$ and this point yields the point $(Q,R,S)$ on the surface $S_{1}:\;F_{1}(Q,R,S)=0$. Because we are interested in finding $a, b, c$ such that (\ref{finitesys}) has infinitely many solutions we want to find instances of $a, b, c$ such that the curve $C_{1}$ is singular. In order to do so, we compute the Gr\"{o}bner basis of the ideal $I_{1}'=<F_{2}(R,S),\partial_{R}F_{2}(R,S),\partial_{S}F_{2}(R,S)>$ with respect to the order $R<S<a<b<c$. Next, we observe that
\begin{equation*}
\op{GB}(I_{1}')\cap\Z[a,b,c]=\{(a^3-4 a b+8 c) (-9 a^2 b^2+32 b^3+27 a^3 c-108 a b c+108 c^2)\}.
\end{equation*}
We thus put $c=\frac{a(a^2-4b)}{8}$. After this substitution we notice that the curve $C_{1}$ has two singular points
\begin{equation*}
(R,S)=\Big(\frac{1}{4}(a\pm\sqrt{3a^2-8b}),\frac{1}{4}(a\mp\sqrt{3a^2-8b})\Big).
\end{equation*}
In order to make these points rational we substitute $b=\frac{3a^2-d^2}{8}$, where $d$ is rational parameter. This leads us to
\begin{equation*}
b=\frac{1}{8}(3a^2 - d^2),\quad c=\frac{1}{16}a(a^2 - d^2).
\end{equation*}
We note that this choice leads to
\begin{equation*}
F_{2}(R,S)=(2R+2S-a)(a^2 - d^2 - 4 aR + 8R^2 - 4 aS+ 8S^2).
\end{equation*}
The condition $2R+2S-a=0$ leads to trivial solutions. However, the equation  $a^2 - d^2 - 4 aR + 8R^2 - 4 aS+ 8S^2=0$ defines a genus 0 curve, say $\cal{C}$, with a rational point $(R,S)=\Big(\frac{a-d}{4},\frac{a+d}{4})$ and thus it can be parameterized by rational functions. The parametrization has the following form:
\begin{equation*}
R=\frac{a + d - 2 d t + (a - d) t^2}{4(1+t^2)},\quad S=\frac{a+d+2dt+(a-d)t^2}{4(1+t^2)}.
\end{equation*}
With the $R, S$ given above we easily solve the system (\ref{finitesys}) getting
\begin{equation*}
P=\frac{a - d + 2 d t + (a + d) t^2}{4(1+t^2)},\quad Q=\frac{a - d - 2 d t + (a + d) t^2}{4(1+t^2)}.
\end{equation*}
We thus see that for each rational $t$ and $a, d$ with $a(a^2-d^2)\neq 0$ we get infinitely many non-trivial solutions of the system (\ref{finitesys}).

\bigskip

In order to tackle the system (\ref{finitesys}) with $(i_{1},i_{2},i_{3})=(1,3,4)$ we use the same approach as in the case $(i_{1},i_{2},i_{3})=(1,2,3)$.
We compute the Gr\"{o}bner basis of the ideal, say $I_{2}$, containing the equations defining (\ref{finitesys}) with the order $P<Q<R<S$ and note that $\op{GB}(I_{2})\cap \Z[R,S]=\{G(R,S)\}$, where
\begin{equation*}
G(R,S)=b R S - a R^2 S^2 - c (R + S) + R^2 S^2 (R + S).
\end{equation*}
Computation of the Gr\"{o}bner basis of $I_{2}'=<G(R,S),\partial_{R}G(R,S),\partial_{S}G(R,S)>$ leads us to the intersection
\begin{equation*}
\op{GB}(I_{2}')\cap\Z[a,b,c]=\{c^2 (-b^2+a^2 c) (-a^3 b^3+27 b^4-6 a^2 b^2 c+27 a^4 c^2-768 a b c^2+4096 c^3)\}.
\end{equation*}
We thus put $c=\frac{b^2}{a^2}$ and then we get
\begin{equation*}
G(R,S)=(b-a R S)(-bR-bS+a^2RS-aR^2S-aRS^2)
\end{equation*}
and thus we put
\begin{equation}\label{S}
S=\frac{b}{aR}
\end{equation}
 with $a\neq 0$. Solving now the first equation from the system (\ref{finitesys}) with respect to $P$ we get
 \begin{equation}\label{P}
 P=\frac{-b+(a^2-aq)R-aR^2}{aR}.
 \end{equation}
Putting the expressions for $P, S$ into the second and third equation from the system (\ref{finitesys}) we get
\begin{align*}
&\sigma_{3}(P,Q,R,S)-a=-\frac{(a R^2+b)(bQ+(a Q^2-a^2Q+b)R+aQR^2)}{a^2 R^2},\\ &\sigma_{4}(P,Q,R,S)-c=\frac{b(bQ+(a Q^2-a^2Q+b)R+aQR^2)}{a^2R}.
\end{align*}
We are left with solving the equation $bQ+(a Q^2-a^2Q+b)R+aQR^2=0$ or equivalently the equation defining the curve $C_{2}$ given by
\begin{equation*}
C_{2}:\;V^2=(aR^2-a^2R+b)^2-4 abR^2,
\end{equation*}
where we put $V=2aQR+b-a^2R+aR^2$. In order to guarantee that $C_{2}$ will have infinitely many rational points we put
\begin{equation*}
b=-\frac{(a^4 - d^2)^2}{16 a d^2},
\end{equation*}
where $d$ is a parameter. Then on $C_{2}$ we have two points:
\begin{equation*}
U_{1}=\Big(0,-\frac{(a^4 - d^2)^2}{16 a d^2}\Big)\quad U_{2}=\Big(-\frac{\left(a^2-t\right)^2}{4 a t},\frac{\left(a^2-t\right)^3 \left(a^2+t\right)}{8 a t^2}\Big).
\end{equation*}
Using the point $U_{1}$ as the point at infinity we get that $C_{2}$ is birationally equivalent with the curve in the Weierstrass form $E_{2}:\;Y^2=X^3+AX+B$, where
\begin{align*}
&A=-27 (a^{16} + 12 a^{12} d^2 - 10 a^8 d^4 + 12 a^4 d^6 + d^8),\\
&B=-54 (a^8 + d^4) (a^{16} - 36 a^{12} d^2 + 38 a^8 d^4 - 36 a^4 d^6 + d^8).
\end{align*}
We can look on $E_{2}$ as on the elliptic curve defined over $\Q(a,d)$. Now we have $\Delta (E_{2})=-2^{4}3^{12}a^4 (a^2 - d)^8 d^2 (a^2 + d)^8 (a^8 - a^4 d^2 + d^4)$.  Moreover, the point $U_{2}$ lying on $C_{2}$ corresponds to the point
\begin{equation*}
\phi(U_{2})=U=(-3 (a^8 + 6 a^6 d - 6 a^2 d^3 - 5 d^4), 54d^2(a^2 - d)(a^2 + d)^3),
\end{equation*}
where $\phi:\;C_{2}\rightarrow E_{2}$ is a birational map. Because the square of the $Y$-th coordinate of $V$ does not divide $\Delta (E_{2})$, we immediately deduce that $V$ is of infinite order on $E_{2}$. We thus see that in order to find infinitely many solutions of the system (\ref{finitesys}) with chosen $a, b, c$ we first compute $[k]U$ for $k=1,2,\ldots $; next we compute the corresponding $\phi^{-1}([k]U)=(R_{k}, V_{k})$ lying on $C_{2}$. Then from the expressions for $P, S$ and $V$ we get
\begin{align*}
&Q_{k}=\frac{a^8-2 a^4 d^2+16 a^3 d^2 R_{k}-16 a^2 d^2 R_{k}^2+16 a d^2 V_{k}+d^4}{32 a^2 d^2 R_{k}},\\
&P_{k}=\frac{-b+(a^2-aQ_{k})R_{k}-aR_{k}^2}{aR_{k}},\\
&S_{k}=-\frac{(a^4 - d^2)^2}{16 a d^2aR_{k}}.
\end{align*}
We thus obtain the solutions of the system (\ref{finitesys}) for $k=1,2,\ldots.$ Our theorem is proved.

\end{proof}

The well-known interpretation of symmetric polynomials as coefficients in the expansion of the polynomial $\prod_{i=1}^{n}(X+x_{i})$ implies the following:

\begin{cor}\label{cor:quartic}
There is an infinite set $\cal{A}_{i_{1},i_{2},i_{3}}$ for $(i_{1},i_{2},i_{3})=(1,2,3), (1,3,4)$ of triples of rational numbers $(a,b,c)$ such that for any $(a,b,c)\in \cal{A}_{i_{1},i_{2},i_{3}}$ there exists infinitely many rational numbers $t$ such that the polynomial
$X^4+aX^3+bX^2+cX+t$ (in case of $(i_{1},i_{2},i_{3})=(1,2,3)$) has only rational roots. Similar property holds for the polynomials of the form $X^4+aX^3+tX^2+bX+c$ in case of $(i_{1},i_{2},i_{3})=(1,3,4)$.
\end{cor}

We end with the following observation.

\begin{thm}\label{thm5}
Let $n\geq 5$. Then there exist infinitely many triples of rational numbers $a,b,c$ such that
the system of Diophantine equations
\begin{equation}\label{threesystem}
\sigma_{1}(x_{1},\ldots, x_{n})=a,\quad \sigma_{2}(x_{1},\ldots, x_{n})=b,\quad \sigma_{3}(x_{1},\ldots, x_{n})=c
\end{equation}
has infinitely many solutions in rational numbers $(x_{1},\ldots,x_{n})$.
\end{thm}
\begin{proof}
In order to tackle the system (\ref{threesystem}) let us put $\overline{T}=(t_{1},\ldots,t_{n-1})$, where $t_{1},\ldots, t_{n-1}$ are rational parameters and
\begin{equation*}
a=\sigma_{1}(\overline{T},t_{1}+t_{2}-t_{n-1}),\quad b=\sigma_{2}(\overline{T},t_{1}+t_{2}-t_{n-1}),\quad c=\sigma_{3}(\overline{T},t_{1}+t_{2}-t_{n-1}).
\end{equation*}
We put now
\begin{equation}\label{solex}
x_{1}=P,\;x_{2}=Q,x_{3}=t_{1}+t_{2}-Q,\;x_{4}=t_{1}+t_{2}-P,\;x_{i}=t_{i-2}\quad\mbox{for}\quad i=5,\ldots,n.
\end{equation}
For $x_{i}$ defined in this way let us put $\overline{X}=(x_{5},\ldots,x_{n})$. We are now interested only in the variables $P, Q$. We see that the equality $\sigma_{1}(P,Q,x_{3},x_{4},\overline{X})=a$ holds true. Now, we note the following identities
\begin{equation}\label{sigma2}
\sigma_{2}(P,Q,x_{3},x_{4},\overline{X})-b=t_{n-1}^2+(t_{1}+t_{2})(P+Q-t_{n-1})-t_{1}t_{2}-(P^2+Q^2),
\end{equation}
and
\begin{equation}\label{sigma3}
\sigma_{3}(P,Q,x_{3},x_{4},\overline{X})-c=\Big(\sum_{i=1}^{n-2}t_{i}\Big)(\sigma_{2}(P,Q,x_{3},x_{4},\overline{X})-b).
\end{equation}
We present only the proof of the first identity. Using exactly the same method of expansion of  $\sigma_{2}(P,Q,x_{3},x_{4},\overline{X})$ as on the beginning of the proof of Theorem \ref{thm1} we get
\begin{equation*}
\sigma_{2}(P,Q,x_{3},x_{4},\overline{X})=\sigma_{2}(P,Q,x_{3},x_{4})+(P+Q+x_{3}+x_{4})\sigma_{1}(\overline{X})+\sigma_{2}(\overline{X}).
\end{equation*}
We note now the following identities
\begin{align*}
\sigma_{2}(P,Q,x_{3},x_{4})&=(t_1+t_2)(P+Q)-P^2-Q^2+(t_{1}+t_{2})^2,\\
            P+Q+x_{3}+x_{4}&=2(t_{1}+t_{2}),\\
 \sigma_{2}(\overline{X})-b&=t_{n-1}^2-t_{1}t_{2}-(t_{1}+t_{2})(2\sigma_{1}(\overline{T})-t_{1}-t_{2}-t_{n-1}).
\end{align*}
The last identity follows from the symmetry and expansion of $\sigma_{2}$ with respect to $t_{1}$ and $t_{2}$. Now a simple calculation of the expression on the left in (\ref{sigma2}) gives us the expression on the right. The second identity can be proved analogously. This is a simple but tedious calculation based on the expression for $\sigma_{3}(P,Q,x_{3},x_{4},\overline{X})$ with $P, Q, x_{3}, x_{4}$ written explicitly. We leave the simple details to the reader.

In the light of our computations we see that it is enough to show that the quadric
\begin{equation*}
\cal{Q}:\;t_{n-1}^2+(t_{1}+t_{2})(P+Q-t_{n-1})-t_{1}t_{2}-(P^2+Q^2)=0
\end{equation*}
defined over the field $\Q(t_{1},t_{2},t_{n-1})$ has infinitely many rational points. However, one can check that $\cal{Q}$ contains the point $(P,Q)=(t_{1},t_{n-1})$ and this implies that $\cal{Q}$ can be parameterized by rational functions. The parametrization is given by
\begin{equation*}
P=\frac{t_1 u^2+(t_1+t_2-2t_{n-1})u+t_2}{u^2+1},\quad Q=\frac{(t_{1}+t_{2}-t_{n-1})u^2+(t_2-t_1)u+t_{n-1}}{u^2+1}.
\end{equation*}
We have thus proved that the system (\ref{threesystem}) with $a, b, c$ defined above has infinitely many solutions in rational numbers $x_{1},\ldots, x_{n}$, where $x_{i}$ is given by (\ref{solex}) and $P,Q$ are given above. In fact our method gives a parametric solution of the system under consideration.
\end{proof}

\begin{exam}\label{exam3}
{\rm In order to see an example of of solutions of the system (\ref{threesystem}) we take $n=5$ and $(t_{1},t_{2},t_{3},t_{4})=(1,1,1,2)$. We thus consider the system
\begin{equation*}
\sigma_{1}(x_{1},\ldots, x_{5})=5,\quad \sigma_{2}(x_{1},\ldots, x_{5})=9,\quad \sigma_{3}(x_{1},\ldots, x_{5})=7.
\end{equation*}
Our method give us the following solution of this system
\begin{equation*}
(x_{1},x_{2},x_{3},x_{4},x_{5})=\Big(\frac{(u-1)^2}{u^2+1},\frac{2}{u^2+1},\frac{(u+1)^2}{u^2+1},\frac{2 u^2}{u^2+1},1\Big)
\end{equation*}
in positive rational functions. As a curiosity let us also note that $\sigma_{5}(x_{1},\ldots, x_{5})$ is a square of a rational function

}
\end{exam}

From the result above we easily deduce the following complement of Corollary \ref{cor:quartic}.

\begin{cor}
There exists an infinite set $\cal{G}$ of triples of rational numbers with the property that if $(a, b, c)\in\cal{G}$, then there are infinitely many rational $(n-3)$-tuples $(t_{4},\ldots,t_{n})$ such that all the roots of the polynomial
\begin{equation*}
X^n+aX^{n-1}+bX^{n-2}+cX^{n-3}+\sum_{i=0}^{n-4}t_{i}X^{i}
\end{equation*}
are rational.
\end{cor}

In the light of the result above one can ask an interesting:

\begin{ques}
Let $n\geq 5$ and $1\leq i_{1}<i_{2}<i_{3}\leq n$ be given. Is it possible to find rational numbers $a, b, c$ such that the system
\begin{equation*}
\sigma_{i_{1}}(x_{1},\ldots, x_{n})=a,\quad \sigma_{i_{2}}(x_{1},\ldots, x_{n})=b,\quad \sigma_{i_{3}}(x_{1},\ldots, x_{n})=c
\end{equation*}
has infinitely many rational solutions? In case $i_{3}=n$ we assume that $c\neq 0$.
\end{ques}

We hope to come back to this and some related questions in a near future.

\bigskip

\noindent {\bf Acknowledgments}
The author express his gratitude to the referee for a careful reading of the manuscript
and many valuable suggestions, which improve the quality of this paper. The research of the author was supported by Polish Government funds for science, grant IP 2011 057671 for the years 2012--2013.

\bigskip
\noindent Jagiellonian University, Faculty of Mathematics and Computer Science, Institute of Mathematics, {\L}ojasiewicza 6, 30 - 348 Krak\'{o}w, Poland;
 email: {\tt maciej.ulas@uj.edu.pl}

 \end{document}